\documentclass[12pt]{amsart}
\usepackage{a4wide}
\usepackage[utf8]{inputenc}
\usepackage{amsmath}
\usepackage{amsthm}
\usepackage{amssymb}
\usepackage{amsfonts}
\usepackage{amsopn}
\usepackage{graphicx}
\usepackage{enumerate}
\usepackage{color}
\usepackage{mathtools}
\usepackage{mathrsfs}
\usepackage{bbm}
\usepackage{yhmath}
\usepackage{cite}
\usepackage[colorlinks,linkcolor=blue,anchorcolor=blue,citecolor=blue]{hyperref}

\newtheorem{theorem}{Theorem}[section]

\newtheorem{lemma}[theorem]{Lemma}
\newtheorem{definition}[theorem]{Definition}
\newtheorem{corollary}[theorem]{Corollary}

\theoremstyle{definition}
\newtheorem{example}[theorem]{Example}

\newtheorem*{conjecture}{Conjecture} % no numbers

\newtheorem*{remark}{Remark}

\numberwithin{equation}{section}

\newcommand{\mbb}{\mathbb}

\newcommand{\mbf}{\mathbf}
\newcommand{\mcal}{\mathcal}
\newcommand{\mscr}{\mathscr}
\newcommand{\set}[1]{\left\{ #1 \right\}}

\newcommand{\R}{\mathbb{R}}

\newcommand{\Z}{\mathbb{Z}}
\newcommand{\N}{\mathbb{N}}

\newcommand{\f}{\infty}

\newcommand{\wh}[1]{\widehat{#1}}

\newcommand{\sse}{\subseteq}

\newcommand{\D}{\;\mathrm{d}}

\newcommand{\blue}[1]{{\color{blue}#1}}

\title[Existence, equivalence and spectrality of infinite convolutions in $\R^d$  ]{Existence, equivalence and spectrality of infinite convolutions in $\R^d$  }

\author[J. J. Miao]{Jun Jie Miao}
\address[J. J. Miao]{School of Mathematical Sciences,  Key Laboratory of MEA (Ministry of Education) \& Shanghai Key Laboratory of PMMP,  East China Normal University, Shanghai 200241, China}
\email{jjmiao@math.ecnu.edu.cn}

\author[H. Zhao]{Hongbo Zhao}
\address[H. Zhao]{School of Mathematical Sciences, Shanghai Key Laboratory of PMMP, East China Normal University, Shanghai 200241,
	People's Republic of China}
\email{2504245357@qq.com}

\begin{document}
\maketitle

\begin{abstract}	
In this paper, we study existence, equivalence and spectrality of infinite convolutions  which may not be compactly supported in $d$-dimensional Euclidean space by manipulating various techniques in probability theory. First, we define the equivalent sequences, and we prove that the infinite convolutions converges simultaneously if they are generated by equivalent sequences. Moreover, the equi-positivity keeps unchanged for  infinite convolutions generated by equivalent sequences. Next, we study the spectrality of  infinite convolutions generated by admissible pairs, and we show such infinite convolutions have the same spectrum if they are generated by the equivalent sequences. Finally, we  provide some  sufficient conditions for the existence and spectral properties of infinite convolutions in higher dimensions.
\end{abstract}

\section{Introduction}

\subsection{Infinite convolutions and non-autonomous affine measures}

For a finite subset $A \sse \R^d$, the uniform discrete measure supported on $A$ is given by
\begin{equation*}%\label{def-delta}
	\delta_A = \frac{1}{\# A} \sum_{a \in A} \delta_a,
\end{equation*}
where  $\#$ denotes the cardinality of a set and $\delta_a$ denotes the Dirac measure at the point $a$. Let $\{ A_k\}_{k=1}^\f$ be a sequence of finite subsets of $\R^d$ such that $\# A_k \ge 2$ for every $k \ge 1$.
For  each integer $k \geq 1$, we define
\begin{equation}\label{discrete-convolution}
	\nu_k =\delta_{A_1}*\delta_{A_2} * \cdots *\delta_{A_k},
\end{equation}
where $*$ denotes the convolution of measures.
If the sequence of convolutions $\{\nu_k\}_{k=1}^\f$ converges weakly to a Borel probability measure $\nu$, then we call $\nu$ the \emph{infinite convolution} of $\{{A_k}\}_{k=1}^\infty$, denoted by
\begin{equation}\label{def_ica}
	\nu =\delta_{A_1}*\delta_{A_2} * \cdots *\delta_{A_k} *\cdots.
\end{equation}

There is a  famous example of infinite convolutions
$$
\nu_\lambda = \delta_{\{\pm \lambda\}} * \delta_{\{\pm \lambda^2\}} * \cdots *\delta_{\{\pm \lambda^k\}}*\cdots, \qquad \lambda\in(0,1),
$$
 called \textit{infinite Bernoulli convolution}, that is, $A_k=\{-\lambda^k,\lambda^k\}$ in \eqref{discrete-convolution}.
The measure $\nu_\lambda$ may be also seen as the distribution of $\sum _{k=1}^\infty \pm \lambda^k$ where the signs are  independently selected with probability $\frac{1}{2}$. Infinite Bernoulli convolutions have been studied  since 1930's for the fundamental question to determine the absolute continuity and singularity of $\nu_\lambda$ for $\lambda \in (\frac{1}{2}, 1)$. Moreover,  the smoothness of the density is explored if $\nu_\lambda$ is absolutely continuous, and the dimension of the measure is investigated  if $\nu_\lambda$ is singular with respect to the Lebesgue measure.
These studies reveal connections with  harmonic analysis, the theory of algebraic numbers, dynamical systems, and fractal geometry \cite{Chan-Ngai-Teplyaev-2015,Hu-Lau-2019,Shmerkin-2014,Solomyak1995}, we refer readers to the  excellent survey paper \cite{Peres-Schlag-Solomyak-2000} for this topic.

It is clear that the uniformly distributed self-affine measures and non-autonomous affine measures may be regarded as special cases of infinite convolutions; see \cite{Falco03,GM22,GM}. Given a sequence $\{(R_k,B_k)\}_{k=1}^\infty $where $R_k$ is a $d \times d$ expansive matrix (all eigenvalues have modulus strictly greater than $1$) with integer entries and $B_k\sse \Z^d$ is a finite subset of integer vectors with $\# B_k\ge 2$ for each integer $k\geq 1$. We write
\begin{equation}\label{def_mun}
	\mu_k =\delta_{{R_1}^{-1}B_1}\ast\delta_{(R_2R_1)^{-1}B_2}\ast\dots\ast\delta_{(R_kR_{k-1}\cdots R_1)^{-1}B_k}.
\end{equation}
If the sequence $\{\mu_k\}_{k=1}^\infty$ converges weakly to a Borel probability measure $\mu$, then we call $\mu$ the \emph{infinite convolution} of $\{(R_k,B_k)\}_{k=1}^\infty,$ denoted by
\begin{equation}\label{infinite-convolution}
	\mu =\delta_{{R_1}^{-1}B_1}\ast\delta_{(R_2R_1)^{-1}B_2}\ast\dots\ast\delta_{(R_kR_{k-1}\cdots R_1)^{-1}B_k} *\cdots.
\end{equation}
For each integer $k\geq 1$, we write
\begin{equation}\label{def_mugn}
	\mu_{>k}=\delta_{({R_{k+1}\cdots R_1})^{-1}B_{k+1}}\ast\delta_{(R_{k+2}\cdots R_1)^{-1}B_{k+2}}\ast\cdots,
\end{equation}
and it is clear that $\mu = \mu_k * \mu_{>k}$. Let
\begin{equation}\label{def_nu_n}
	\nu_{>k} = \mu_{>k}\circ(R_kR_{k-1}\cdots R_1)^{-1}=\delta_{R_{k+1}^{-1} B_{k+1}} * \delta_{(R_{k+2} R_{k+1})^{-1} B_{k+2}} * \cdots,
\end{equation}
which is crucial to investigate  the spectrality of infinite convolutions (see Theorem \ref{thm_equi_spectral}).

Note that if all $(R_k,B_k)$ are identical, the corresponding  infinite convolution is just a self-similar measure or a self-affine measure. For the non-autonomous affine measures, it usually requires  the {\it uniform contractive condition}, that is,
$$
\sup_{k} \Vert R_k^{-1}\Vert<1,
$$
where $\Vert\cdot\Vert$ is the spectral norm, otherwise the corresponding non-autonomous affine sets may have interior points; see \cite{Falco03,GM22,GM} for details.

\subsection{Spectral measures and fractals}
A Borel probability measure $\mu$ on $\R^d$ is called a \emph{spectral measure} if there exists a countable subset $\Lambda \sse \R^d$ such that the family of exponential functions
$$
\set{e_\lambda(x) = e^{-2\pi i \lambda\cdot x}: \lambda \in \Lambda}
$$
forms an orthonormal basis in $L^2(\mu)$ where $\lambda\cdot x$ is the standard inner product in $\R^d$, and we call the set $\Lambda$   a \emph{spectrum} of $\mu$ and call $(\mu, \Lambda)$ a {\em spectral pair}.
The existence of spectrum of measures is a fundamental question in harmonic analysis, which was first studied by Fuglede~\cite{Fuglede-1974} for the normalized Lebesgue measure on measurable sets.
In the paper, Fuglede proposed the following  conjecture.
\begin{conjecture}
	\emph{A measurable set $\Gamma \sse \R^d$ with positive finite Lebesgue measure is a spectral set, that is, the normalized Lebesgue measure on $\Gamma$ is a spectral measure, if and only if $\Gamma$ tiles $\R^d$ by translations.}
\end{conjecture}
Although the conjecture has been disproved for $d \ge 3$, it remains open for $d=1$ and $d=2$. We refer readers to  \cite{Farkas-Matolcsi-Mora-2006,Farkas-Revesz-2006,Kolountzakis-Matolcsi-2006a,Kolountzakis-Matolcsi-2006b,Matolcsi-2005, Tao-2004} for related studies and references therein. However the connection between spectrality and tiling has raised huge interest, and some affirmative results have been proved for special cases \cite{Iosevich-Katz-Tao-2003,Laba-2001}.

Fractal measures are singular with respect to Lebesgue measures; see~\cite{Falco03} for the background of fractal geometry.  There is a big difference in the theory of spectrum between absolutely continuous measures and singular measures. In particular,  Jorgensen and Pedersen~\cite{Jorgensen-Pedersen-1998}  found that some self-similar measures may have spectra. For example, let $\mu$ be the self-similar measure on $\R$ given by
$$
\mu(A) = \frac{1}{2} \mu(S_1^{-1}(A)) + \frac{1}{2} \mu(S_2^{-1}(A)),
$$
for all Borel $A\subset\R$, where $S_1(x)=\frac{1}{4}x$ and $S_2(x)=\frac{1}{4}x+\frac{1}{2}$.
Then $\mu $ is a spectral measure with a spectrum $$\Lambda = \bigcup_{n=1}^\f \set{\ell_1 + 4\ell_2 + \cdots + 4^{n-1} \ell_n: \ell_1,\ell_2,\cdots,\ell_n \in \set{0,1}}.$$
From then on, the spectrality of various fractal measures, such as self-similar measures, self-affine measures and non-autonomous affine measures, has been extensively studied. See \cite{An-Fu-Lai-2019,An-He-He-2019,An-He-2014,An-Lai-2020, An-Wang-2021,Dai-He-Lau-2014,Deng-Chen-2021, Dutkay-Haussermann-Lai-2019,Laba-Wang-2002,Li-2009,Liu-Dong-Li-2017,
	He-Tang-Wu-2019} and references therein.

Admissible pairs are the key to study the spectrality of infinite convolutions. Given a $d \times d$ expansive integral matrix $R$ (all eigenvalues have modulus strictly greater than $1$) and a finite subset  $B\sse \Z^d$ with $\# B\ge 2$.  If there exists $L\sse \Z^d$ such that the matrix
$$
\left[ \frac{1}{\sqrt{\# B}} e^{-2 \pi i  (R^{-1}b)\cdot\ell }  \right]_{b \in B, \ell \in L}
$$
is unitary, we call $(R, B)$ an {\it admissible pair} in $\R^d$, we also call $(R,B,L)$ a {\it Hadamard triple} in $\R^d$. See~\cite{Dutkay-Haussermann-Lai-2019} for details.
	
In \cite{Laba-Wang-2002}, {\L}aba and Wang proved that self-similar measures generated by admissible pairs  with equal weights in $\R$ are  spectral measures. In~\cite{Dai-He-Lau-2014}, Dai, He and Lau completely resolved the spectrality  of the self-similar measures generated by admissible pairs  $(N, B)$ where $B=\{0,1,\ldots, b-1\}$ and $b | N$. In 2019, Dutkay, Haussermann and Lai \cite{Dutkay-Haussermann-Lai-2019} generalized {\L}aba and Wang's work to higher dimensional case and proved that self-affine measures generated by admissible pairs with equal weights in $\R^d$ are  spectral measures.

It is natural to investigate the spectrality of infinite convolutions since  fractal spectral measures may be represented by infinite convolutions. The spectrality of the infinite convolution generated by admissible pairs was first studied by Strichartz~\cite{Strichartz-2000}. Since  then, huge interest has been aroused by the spectral question on infinite convolutions.

In 2017, Dutkey and Lai \cite{Dutkay-Lai-2017} proved that the infinite convolution $\mu$ is spectral if $\mu$ is compactly supported and satisfies no-overlap condition, and if the infimum of the singular values of $\frac{1}{\#  B_n}[|\hat{\mu}_{>n}(l)|e^{-2\pi i<b,l>}]_{l\in L_n, b\in B_n}$ is positive.
In 2019, An, Fu and Lai  \cite{An-Fu-Lai-2019} studied this question in $\R$ and proved that the infinite convolution $\mu$ is spectral if $B_n\subset\{0,1,\ldots,N_n-1\}$ for all $n\geq 1$ and $\liminf_{n\to \infty} \#  B_n <\infty$ where the support of $\mu$ is a compact subset of $[0,1]$.  We refer  readers to ~\cite{An-Fu-Lai-2019,An-He-Lau-2015,An-He-Li-2015,Dai-2012, Dutkay-Haussermann-Lai-2019,Dutkay-Lai-2017,Laba-Wang-2002} for further studies on infinite convolutions.

Currently, the spectral theory of infinite convolutions mainly focuses on the measures with compact support. In this paper, we study the infinite convolutions generated by infinitely many admissible pairs which may not be compactly  supported in $\R^d$.

\iffalse
In 2022, Li, Miao and Wang \cite{Miao-2022} proved that the infinite convolution $\mu$ of $\{N_k,B_k\}_{k=1}^\f$ is a spectral measure supported on a non-compact set, where $B_k\equiv\{0,1,\cdots,b_k\}\pmod {N_k}$, $b_k|N_k$ for all $k>0$ and
$$
\sum_{k=1}^{\f}\frac{1}{b_k}\# \{B_k\setminus  \{0,1,\cdots,b_k\}\}<\f.
$$ \blue{
The elements of $B_k$ are not strictly required here, which means that small perturbation of this particular measure dose not change its spectrality. Miao and Zhao \cite{MZ24} give a more general condition under which the spectral property of infinite convolutions may be preserved in some cases. We make some generalizations of this condition. In this paper, we give a sufficient condition for preserving the spectral property.
}

Our main conclusions are stated in Section~\ref{sec_mr},  and we provide some prerequisites and lemmas in Section~\ref{sec_pre}. In Section~\ref{sec_equ}, we investigate the properties of equivalent sequences, prove that infinite convolutions generated by equivalent sequences are consistent in existence and equi-positivity, and provide sufficient conditions for them to have the same spectrum. In Section~\ref{sec_existence} and Section~\ref{sec_spectral}, we studied the existence and spectrality of infinite convolutions in higher dimensions, giving a technical condition that proves that infinite convolutions are spectral measures.
\fi

\section{Main results and Examples}\label{sec_mr}
In this section, we state the key definitions and our main conclusions. First, we give the definition of equivalence of two sequences consisting of sets, which we use later to reveal the connections between the infinite convolutions with compact support and the ones without compact support.

\begin{definition}\label{def_equivalent}
Two sequences of finite sets $\{A_k\}_{k=1}^\infty$ and $\{A'_k\}_{k=1}^\infty$ in $\R^d$ are called {\it equivalent} if
$$
\sum_{k=1}^{\f}\Big(\frac{\#(A'_k\setminus A_k)}{\# A'_k}+\frac{\#(A_k\setminus A'_k)}{\# A_k}\Big)<\f.
$$
\end{definition}
Obviously, $\{A_k\}_{k=1}^\infty$ and $\{A'_k\}_{k=1}^\infty$ are equivalent if and only if
$$
\sum_{k=1}^{\f}\max\bigg\{\frac{\#(A'_k\setminus A_k)}{\# A'_k},\frac{\#(A_k\setminus A'_k)}{\# A_k}\bigg\}<\f.
$$

Our first conclusion  shows that the infinite convolutions generated by equivalent sequences converge simultaneously.
\begin{theorem}\label{thm_equivalent}
Given two equivalent sequences $\{A_k\}_{k=1}^\infty$ and $\{A'_k\}_{k=1}^\infty$ in $\R^d$. Then the infinite convolution $\nu$ of $\{A_k\}_{k=1}^\infty $ given by \eqref{def_ica}  exists if and only if the infinite convolution $\nu'$ of $\{A_k'\}_{k=1}^\infty $ exists.
\end{theorem}

In this paper, we characterize the spectral property of infinite convolutions through  equi-positive measures (see Definition \ref{def_equi_measure}). The next conclusion shows that the infinite convolutions generated by equivalent sequences are equi-positive simultaneously.
\begin{corollary}\label{cor_equivalent}
Let $\{(R_k,B_k)\}_{k=1}^\infty $ and $\{(R_k,B'_k)\}_{k=1}^\infty $ be two sequences in $\R^d$. Suppose that $\{B_k\}_{k=1}^\infty$ and $\{B'_k\}_{k=1}^\infty$ are equivalent. Then the infinite convolution $\mu$ of $\{(R_k,B_k)\}_{k=1}^\infty $ given by \eqref{infinite-convolution} exists if and only if the infinite convolution $\mu'$ of $\{(R_k,B'_k)\}_{k=1}^\infty $ exists.
	
Moreover, if $\mu$ exists, then $\mu$ is an equi-positive measure if and only if $\mu'$ is an equi-positive measure.
\end{corollary}

It is interesting to study the spectrality of infinite convolutions with non-compact support. The following theorem provides a new method to construct various infinite convolutions with non-compact support. Moreover, it also shows that  infinite convolutions generated by equivalent sequences have   common spectra. Recall that a sequence $\{(R_k,B_k)\}_{k=1}^\infty $ satisfies uniform contractive condition if $\sup_{k} \Vert R_k^{-1}\Vert<1.$
\begin{theorem}\label{thm_common spectrum}
Given two sequences of admissible pairs  $\{(R_k,B_k)\}_{k=1}^\infty $ and $\{(R_k,B'_k)\}_{k=1}^\infty $  satisfying uniform contractive condition. Suppose that $\{B_k\}_{k=1}^\infty$ and $\{B'_k\}_{k=1}^\infty$ are equivalent and $B'_k\equiv B_k \pmod{R_k\Z^{d}}$ for all $k>0$.
If the infinite convolution $\mu$ of $\{(R_k,B_k)\}_{k=1}^\infty $ exists and is an equi-positive measure, then  \\
(1) the infinite convolution $\mu'$ of $\{(R_k,B'_k)\}_{k=1}^\infty$ exists and is an equi-positive measure; \\
(2) there exists $\Lambda\subseteq\Z^d$ such that $(\mu,\Lambda)$ and  $(\mu',\Lambda)$ are all spectral pairs.
\end{theorem}

Next, we provide a sufficient condition for the existence of infinite convolutions.  Given a sequence $\{(R_k,B_k)\}_{k=1}^\infty $. We say $\{( R_k,B_k)\}_{k=1}^\infty $ satisfies {\it remainder bounded condition (RBC)} if
\[
\sum_{k=1}^{\infty} \frac{\# B_{k,2}}{\# B_k} < \infty,
\]
where $B_{k,1}=B_k \cap R_k[-\frac{1}{2},\frac{1}{2})^d$ and $B_{k,2}=B_k \backslash B_{k,1}.$ It is a useful condition to study infinite convolutions without compact support; see Example \ref{ex_ncpt}.

% In fact, the sequences $\{(R_k,B_k)\}_{k=1}^\infty $ satisfying RBC are equivalent to a sequence $\{(R_k,B'_k)\}_{k=1}^\infty $ satisfying $B'_k \subset R_k[-\frac{1}{2},\frac{1}{2})^d$ for all $k>0$.
\begin{theorem}\label{thm_RBC_exists}
Let $\{( R_k,B_k)\}_{k=1}^\infty $ be a sequence satisfying {\it RBC} and uniform contractive condition. Then the infinite convolution $\mu$ of $\{( R_k,B_k)\}_{k=1}^\infty $ exists.
\end{theorem}

Although   RBC  and uniform contractive condition are sufficient for the existence of infinite convolution, they are not enough for the spectrality of infinite convolutions.
Given $l\in (0,1)$, we say $\{( R_k,B_k)\}_{k=1}^\infty $ satisfies  \textit{partial concentration condition (PCC)} with $l$ if
\[
\sup\big\{\big(R_k^{-1}(x_1-x_2)\big)\cdot\xi:x_1,x_2\in\bar{B}\big(0,\frac{\sqrt{d}}{2}\big)\big\}<1-l,
\]
for all $\xi\in [-1,1]^d$, and
\[
\sum_{k=1}^{\f}\frac{\# B_{k,2}^{l}}{\# B_k}<\f,
\]
where $B_{k,1}^{l}=B_k\cap \{b:(R_k^{-1}b)\cdot \xi\in(-\frac{1-l}{2},\frac{1-l}{2})\  \text{for all}\ \xi\in[-1,1]^d\}$ and $B_{k,2}^{l}=B_k\backslash B_{k,1}^{l}.$

In the final conclusion, we   show that the  infinite convolution $\mu$ is a spetral measure if the sequence $\{( R_k,B_k)\}_{k=1}^\infty $ also has a subsequence satisfying partial concentration condition (PCC). For $d=1$, the partial concentration condition may be modified to various weaker conditions \cite{MZ24}.
\begin{theorem}\label{thm_PCC_spectral}
Let $\{( R_k,B_k)\}_{k=1}^\infty $ be a sequence of admissible pairs satisfying {\it RBC} and uniform contractive condition.
If there exists a subsequence $\{( R_{n_k},B_{n_k})\}_{k=1}^\infty  $ satisfying  PCC with some $l\in(0,1)$, then the infinite convolution $\mu$ exists and  is a spectral measure with a spectrum in $\Z^d$.

Moreover, if there exists a sequence $\{L_k\}_{k=1}^\infty $ such that $(R_k,B_k,L_k)$ is a Hadamard triple, and $0\in L_k\subseteq R_k^{T}[-\frac{1}{2},\frac{1}{2})^d$ for all $k>0$, then
$$
\Lambda = \bigcup_{k=1}^\f \{L_1+R_1^TL_2+\cdots+(R_{k-1}\cdots R_1)^{T}L_k\}
$$
is a spectrum of $\mu$.
\end{theorem}

In the end,  we provide an example where the  spectral infinite convolution   is  not compactly supported.

\begin{example}\label{ex_ncpt}
For each $k\geq 1$,  let $R_k=diag (8(k+1),8(k+1))$ be a $2\times2$ diagonal matrix,
$$
B_k=\set{(i,j):i,j\in\{0,1,\cdots,k\} \ \text{and} \ (i,j)\ne(k,0)}\cup\{\big(k+8^k(k+1)!,0\big)\},
$$
and $L_k=\{(i,j):i,j\in8\cdot\{0,1,\cdots,k\}-8\cdot t_k\}, $
where $t_k=\frac{k+1}{2}$ if $k$ is odd, and  $t_k=\frac{k}{2}$ if $k$ is even.
Then   infinite convolution $\mu$  of $\{(R_k,B_k)\}_{k=1}^\infty $ exists and is a spectral measure without compact support.

Since  $B_{k,2}=\{\big(k+8^k(k+1)!,0\big)\}$, we have $
\sum_{k=1}^{\infty} \frac{\# B_{k,2}}{\# B_k} = \sum_{k=1}^{\infty} \frac{1}{(k+1)^2}<  \infty,
$
and it implies that  $\{( N_k,B_k)\}_{k=1}^\infty $ satisfies {\it RBC}. By Theorem \ref{thm_RBC_exists}, the infinite convolution $\mu$ exists.
	
It is clear that $\{(R_k,B_k,L_k)\}_{k=1}^\infty $ is a sequence of Hadamard triples and $0\in L_k\subseteq R_k^{T}[-\frac{1}{2},\frac{1}{2})^d$ for all $k>0$. We have
$$
\sup\big\{\big(R_k^{-1}(x_1-x_2)\big)\cdot\xi:x_1,x_2\in\bar{B}\big(0,\frac{\sqrt{2}}{2}\big)\big\}\le\frac{1}{8(k+1))}\cdot \sqrt{2}\cdot|(1,1)|=\frac{1}{4(k+1))},
$$
and
$$
|(R_k^{-1}b)\cdot \xi|\le\frac{1}{8(k+1))}\cdot |(k,k)|\cdot|(1,1)|=\frac{2k}{8(k+1)}<\frac{1}{4},
$$
for all $\xi\in[-1,1]^2$ and $b\in B_k\setminus\{\big(k+8^k(k+1)!,0\big)\}$, it implies that
$\{( R_{k},B_{k})\}_{k=1}^\infty  $ satisfies  PCC with $\frac{1}{4}$.
By Theorem \ref{thm_PCC_spectral}, $(\mu,\Lambda)$ is a spectral pair, where
$$
\Lambda = \bigcup_{k=1}^\f \{L_1+R_1^TL_2+\cdots+(R_{k-1}\cdots R_1)^{T}L_k\}.
$$
	
Since
$$
\bigg|\sum_{k=1}^{\f}(R_kR_{k-1}\cdots R_1)^{-1}\big(k+8^k(k+1)!,0\big)\bigg|=\sum_{k=1}^{\f} \frac{k+8^k(k+1)!}{8^k(k+1)!}=\f,
$$
it is clear that $\mu([0,n]^2)<1$ for all $n\in \N$.  	Therefore the measure $\mu$ is not compactly supported.	
\end{example}

The rest of the paper is organized as follows. In Section \ref{sec_pre}, we recall some definitions and  results, and the properties of equi-positive families are investigated. In Section \ref{sec_equ}, we study the properties of infinite convolutions generated by equivalent sequences,
and give the proofs of Theorem \ref{thm_equivalent}, Corollary \ref{cor_equivalent} and Theorem \ref{thm_common spectrum}. In Section \ref{sec_existence}, we study the existence of infinite convolutions and give the proof of Theorem \ref{thm_RBC_exists}. We study the
spectrality of  infinite convolutions and prove Theorem \ref{thm_PCC_spectral} in the last Section.

\section{Admissible pairs and equi-positive measures}\label{sec_pre}

First, we list some useful properties of admissible pairs; see \cite{Dutkay-Haussermann-Lai-2019, Laba-Wang-2002} for details.

\begin{lemma}\label{lem_hadamard}
	Suppose that $(R,B,L)$ is a Hadamard triple in $R^d$. Then \\	
\noindent	{\rm(i)} $L+l_0$ is a spectrum of $\delta_{R^{-1}B}$ for all $l_0\in\R^d$;
	
\noindent	{\rm(ii)} If $B'\equiv B \pmod{R\Z^d}$ and $L'\equiv L \pmod{R^T\Z^d}$, then $(R,B',L')$ is a Hadamard triple.
	
\noindent	{\rm(iii)} Given a finite sequence $\{(R_j,B_j,L_j)\}_{j=1}^n$ of  Hadamard triples in $\R^d$. Write $\mbf R=R_nR_{n-1}\cdots R_1$,
	$$
	\mbf B=R_nR_{n-1}\cdots R_2 B_1+\cdots+R_nB_{n-1}+B_n,
	$$
	and
	$$
	\mbf L=L_1+R_1^TL_2+\cdots+(R_{n-1}\cdots R_2R_1)^TL_n.
	$$
	Then $(\mbf R,\mbf B,\mbf L)$ is a Hadamard triple.
\end{lemma}

We use $\mcal{P}(\R^d)$ to denote the set of all Borel probability measures on $\R^d$.
For $\mu \in \mcal{P}(\R^d)$, the \emph{Fourier transform} of $\mu$ is given by
$$
\wh{\mu}(\xi) = \int_{\R^d} e^{-2\pi i \xi \cdot x} \D \mu(x).
$$
For a finite set $B\subset \Z^d$, we write $\mathcal{M}_B(\xi)$ for the Fourier transform of the discrete measure $\delta_B$, that is,
\begin{equation}\label{def_FMB}
\mathcal{M}_B(\xi)=\frac{1}{\# B}\sum_{b\in B}e^{-2\pi i b \cdot\xi}.
\end{equation}

Given a Borel probability measure $\mu$ and a subset $\Lambda\subseteq\R^d$, we write
\begin{align*}
	Q_{\mu,\Lambda}(\xi)=\sum_{\lambda\in \Lambda}|\hat{\mu}(\xi+\lambda)|^2,
\end{align*}
where $Q_{\mu,\Lambda}(\xi)=0$ for all $\xi\in\R^d$ if $\Lambda=\emptyset$.
The following theorem is often used to verify the spectrality of measures; see~\cite{Jorgensen-Pedersen-1998} for the proof.
\begin{theorem}\label{thm_Q}%\cite{Jorgensen-Pedersen-1998}
Let $\mu$ be a probability measure on $\R^d$, $\Lambda\subseteq\R^d$. Then the set $\{e^{-2\pi i \lambda\cdot x}:\lambda\in\Lambda\}$ is an  orthonormal basis in $L^2(\mu)$ if and only if $Q_{\mu,\Lambda}(\xi)\equiv1$ for all $\xi\in\R^d$.
\end{theorem}

In general, it is difficult to verify that $Q_{\mu,\Lambda}(\xi)\equiv1$ for the infinite convolution $\mu$ generated by admissible pairs. To this end, An, Fu, Lai \cite{An-Fu-Lai-2019}  came up with the equi-positive family  to verify it, and Dutkay,  Haussermann and  Lai   used the equi-positive family to prove the spectrality of  self-affine measures in \cite{Dutkay-Haussermann-Lai-2019}.

\begin{definition}\label{def_equi_measure}
We call $\Phi\subset\mcal{P}(\R^d)$ an {\it equi-positive family} if there exist reals $\epsilon_0>0$ and $\delta_0>0$ such that for all $x\in[-\frac{1}{2},\frac{1}{2})^d $ and $\mu \in \Phi$, there exists $k_{x,\mu }\in \Z^d$ such that
\[
| \hat{\mu}(x+y+k_{x,\mu })| \ge \epsilon_0,
\]
for all $|y|< \delta_0$, where $k_{x,\mu }=0$ for $x=0$. We sometimes call $\Phi$ an {\it $\epsilon_0$-equi-positive family} to emphasize the dependence of $\epsilon_0$.

We call the infinite convolution $\mu$ of $\{(R_k,B_k)\}_{k=1}^\infty $ an {\it equi-positive measure} if there exists a subsequence $\{\nu_{>n_j}\}_{j=1}^\infty$ (defined by \eqref{def_nu_n}) that forms an equi-positive family.
\end{definition}

The equi-positive family was used to study the spectrality of fractal measures with compact support in \cite{An-Fu-Lai-2019,Dutkay-Haussermann-Lai-2019}, and it was then generalized in~\cite{LMW-2023} to the current version which is also applicable to infinite convolutions without compact support in $\R$. By the same argument, it is straightforward to  extend it  into $\R^d$. Since equi-positive families imply spectrality, which is a useful technique in studying the spectral theory of fractal measures,  we provide the proof here  for the convenience of readers.

\begin{theorem} \label{thm_equi_spectral}
Let $\{(R_k,B_k)\}_{k=1}^\infty $ be a sequence of admissible pairs $\R^d$ satisfying uniform contractive condition. Suppose that the
infinite convolution $\mu$ of  $\{(R_k,B_k)\}_{k=1}^\infty $ exists.
Let $\{\nu_{>k}\}_{k=1}^\infty$ be given by \eqref{def_nu_n}. If  $\{\nu_{>k}\}_{k=1}^\infty$ has a subsequence which is an equi-positive family, then $\mu$ is a spectral measure with a spectrum in $\Z^d$.
\end{theorem}
\begin{proof}
Suppose $\{\nu_{>n_j}\}_{j=1}^\infty$ is a subsequence of  $\{\nu_{>k}\}_{k=1}^\infty$ and also an equi-positive family. Then there exist reals $\epsilon_0>0$ and $\delta_0>0$ such that for all $x\in[-\frac{1}{2},\frac{1}{2})^d $ and $j\ge1$, there exists an integral vector $k_{x,j }\in \Z^d$ such that
\begin{equation}\label{Ineq_epf}
| \hat{\nu}_{>n_j}(x+y+k_{x,j })| \ge \epsilon_0,
\end{equation}
for all $|y|< \delta_0$, where $k_{x,j }=0$ for $x=0$.

For each integer $k>0$, there exists $L_k\subset\Z^d$ such that $(R_k,B_k,L_k)$ is a Hadamard triple since $(R_k,B_k) $ is an admissible pair in $\R^d$. By  (i) in Lemma \ref{lem_hadamard}, we assume that $0\in L_k$ for all $k\ge1$. Given two integers $q>p\ge0$, we write $\mbf R_{p,q}=R_qR_{q-1}\cdots R_{p+1},$
$$
\mbf B_{p,q}=\mbf R_{p,q+1} B_{p+1}+\mbf R_{p,q+2} B_{p+2}+\cdots+\mbf R_{p,q-1} B_{q-1}+B_q,
$$
and
$$
\mbf L_{p,q}=L_{p+1}+\mbf R_{p,p+1}^TL_{p+2}+\cdots+\mbf R_{p,q-1}^TL_{q}.
$$
By  Lemma \ref{lem_hadamard} (i) and (iii) , $\mbf L_{p,q}$ is a spectrum of $\delta_{\mbf R_{p,q}^{-1}\mbf B_{p,q}}$.

Next,  we construct a sequence of finite subsets $\Lambda_j\subset\Z^d$ by induction. For  $j=1$, let $m_1=n_1$ and $\Lambda_1=\mbf L_{0,m_1}$. Note that $\Lambda_1$ is a spectrum of $\mu_{m_1}$ given by  \eqref{def_mun} with $0\in\Lambda_1$.

For $j\ge2$, suppose that $\Lambda_{j-1}$ is a spectrum of $\mu_{m_{j-1}}$ which has been defined with $0\in\Lambda_{j-1}$. Since $\sup_{k} \Vert R_k^{-1}\Vert<1,$
we  choose  $m_j\in \{n_j\}_{j=1}^\infty$ such that $m_j>m_{j-1}$ and
\begin{equation}\label{eq_lambda}
|\mbf R_{0,m_j}^{-T}\lambda|<\frac{\delta_0}{2} \qquad  \textrm{for all $\lambda\in\Lambda_{j-1}$}.
\end{equation}
By \eqref{Ineq_epf}, we choose a vector $k_{\lambda,j}\in\Z^d$ such that
\begin{equation}\label{eq_equi-positive}
| \hat{\nu}_{>n_j}(\mbf R_{m_{j-1},m_j}^{-T}\lambda+y+k_{\lambda,j })| \ge \epsilon_0  \qquad \textrm{for all $|y|<\delta_0$,}
\end{equation}
 where $k_{\lambda,j }=0$ if $\lambda=0$, and we   write
\begin{equation}\label{def_Lmdj}
\Lambda_j=\Lambda_{j-1}+\mbf R_{0,m_{j-1}}^T\{\lambda+\mbf R_{m_{j-1},m_j}^Tk_{\lambda,j}:\lambda\in\mbf L_{m_{j-1},m_j}\}.
\end{equation}
Since $\Lambda_{j-1}$ is a spectrum of $\mu_{m_{j-1}}$, by Lemma \ref{lem_hadamard}(ii) and (iii), we have that $\Lambda_{j}$ is a spectrum of $\mu_{m_j}$. Note that $0\in\mbf L_{m_{j-1},m_j}$, $k_{0,j}=0$ and $0\in\Lambda_{j-1}$, and it is clear that  $0\in\Lambda_j$ and $\Lambda_{j-1}\subseteq \Lambda_j$. Hence the sequence $\{\Lambda_j\}_{j=1}^\infty$ is defined, and we write
$$
\Lambda=\bigcup_{j=1}^\f\Lambda_j.
$$

Finally, it remains to prove that $\Lambda$ is a spectrum of $\mu$, and by Theorem \ref{thm_Q}, it is sufficient to show
$ Q_{\mu,\Lambda}(\xi)=\sum_{\lambda\in \Lambda}|\hat{\mu}(\xi+\lambda)|^2=1 $  for all $\xi\in\R^d$.

For every integer $j\geq 1$, since $\Lambda_{j}$ is a spectrum of $\mu_{m_j}$, by Theorem \ref{thm_Q}, we have that
\begin{equation}\label{eq_Q_j}
Q_{\mu_{m_j},\Lambda_j}(\xi)=\sum_{\lambda\in \Lambda_j}|\hat{\mu}_{m_j}(\xi+\lambda)|^2=1,
\end{equation}
for all $\xi\in\R^d$, and it implies that
$$
\sum_{\lambda\in \Lambda_j}|\hat{\mu}(\xi+\lambda)|^2 =\sum_{\lambda\in \Lambda_j}|\hat{\mu}_{m_j}(\xi+\lambda)|^2|\hat{\mu}_{>m_j}(\xi+\lambda)|^2
\le\sum_{\lambda\in \Lambda_j}|\hat{\mu}_{m_j}(\xi+\lambda)|^2   =1.
$$
It immediately follows that
\begin{equation}\label{eq_Q}
Q_{\mu,\Lambda}(\xi)\le1 \qquad  \textrm {for all $\xi\in\R^d$.}
\end{equation}

Fix $\xi\in\R^d$. We define $f(\lambda)=|\hat{\mu}(\lambda+\xi)|^2 $ for $ \ \lambda\in\Lambda$ and for each $j\ge1$,
$$
f_j(\lambda)=\left\{
\begin{array}{rl}
|\hat{\mu}_{m_j}(\xi+\lambda)|^2, & \text{if}\    \lambda\in\Lambda_j,\\
0\ \ \ \ \ \ \ \ \ \ \ \ \ \ \, & \text{if} \ \lambda\in\Lambda\setminus\Lambda_j.
\end{array}\right.
$$
By \eqref{def_mun} and \eqref{def_mugn},  we immedaitely have 
\begin{equation}\label{def_f_i}
f(\lambda) =|\hat{\mu}_{m_j}(\xi+\lambda)|^2|\hat{\mu}_{>m_j}(\xi+\lambda)|^2 \geq  f_j(\lambda) \big|\hat{\nu}_{>m_j}\big(\mbf R_{0,m_j}^{-T}(\xi+\lambda)\big)\big|^2.
\end{equation}
Note that  $\Lambda_{j-1}\subseteq \Lambda_j$, and it is clear that for each $\lambda\in\Lambda$, there exists $j_\lambda\ge1$ such that $\lambda\in\Lambda_j$ for all $j\ge j_\lambda$. Since $\mu$ is the weak limit of $\mu_j$,  we have $f(\lambda)=\lim_{j\to\f}f_j(\lambda)$.

Since $\sup_{k} \Vert R_k^{-1}\Vert<1,$  there exists an integer $j_0\ge1$   such that for $j\ge j_0$
\begin{equation}\label{eq_xi}
	|\mbf R_{0,m_j}^{-T}\xi|<\frac{\delta_0}{2}.
\end{equation}
Fix $j\ge j_0$. For each $\lambda\in\Lambda_j$, by \eqref{def_Lmdj}, there exist  $\lambda_1\in\Lambda_{j-1}$, $\lambda_2\in\mbf L_{m_{j-1},m_j}$ and $k_{\lambda_2,j}\in \Z^d $ such that
$$
\lambda=\lambda_1+\mbf R_{0,m_{j-1}}^T\lambda_2+\mbf R_{0,m_j}^T k_{\lambda_2,j},
$$
and by  \eqref{eq_lambda} and \eqref{eq_xi}, we have $|\mbf R_{0,m_j}^{-T}(\lambda_1+\xi)|<\delta_0.$
Combining this with \eqref{eq_equi-positive} and \eqref{def_f_i}, we have  that
$$
f(\lambda) \geq f_j(\lambda) \big|\hat{\nu}_{>m_j}\big(\mbf R_{m_{j-1},m_j}^{-T}\lambda_2+\mbf R_{0,m_j}^{-T}(\lambda_1+\xi))+k_{\lambda_2,j}\big)\big|^2 \ge\epsilon_0^2f_j(\lambda).
$$
Therefore, for each $j\ge j_0$, we obtain that $f_j(\lambda)\le\epsilon_0^{-2}f(\lambda) $
for all $\lambda\in\Lambda$.

Let $\rho$ be the counting measure on  $\Lambda$. By \eqref{eq_Q}, it is clear that  $f(\lambda)$ is  $\rho$-integrable, and
$$
Q_{\mu,\Lambda}(\xi)=\sum_{\lambda\in \Lambda}|\hat{\mu}(\xi+\lambda)|^2=\int_{\Lambda} f(\lambda) d\rho(\lambda).
$$
 Since $f_j \le\epsilon_0^{-2} f$ for all $j\geq j_0$,   by the dominated convergence theorem, \eqref{eq_Q_j} and \eqref{def_f_i}, it immediately follows that
$$
Q_{\mu,\Lambda}(\xi) =\lim_{j\to\f}\int_{\Lambda} f_j(\lambda) d\rho(\lambda)
=\lim_{j\to\f}\sum_{\lambda\in \Lambda_j}|\hat{\mu}_{m_j}(\xi+\lambda)|^2 =1.
$$
By Theorem \ref{thm_Q}, the infinite convolution $\mu$ is a spectral measure with spectrum $\Lambda\subseteq\Z^d$.
\end{proof}
\begin{remark}
The spectrum $\Lambda$ of $\mu$ constructed by \eqref{def_Lmdj} purely depends on the equi-positive family $\{\nu_{>n_j}\}_{j=1}^\infty$ in Theorem \ref{thm_equi_spectral}. The spectrum $\Lambda$ is the same as long as $R_k$, $L_k$ and $k_{\lambda,j}$ keep unchanged.
\end{remark}

Equi-positivity is very useful to study the spectrality of measures. In fact, the equi-positivity of a subset in $\mcal{P}(\R^d)$ keeps unchanged if the set is applied with some perturbation.

Given a real $\epsilon>0$ and $\Phi_1,\Phi_2 \subset \mcal{P}(\R^d)$. We say $\Phi_1$ is contained in $\Phi_2$ under {\it $\epsilon$-perturbation}, denoted by
$$
\Phi_1\subseteq_\epsilon\Phi_2,
$$ if for every probability measure $\nu_1\in \Phi_1$, there exists $\nu_2\in\Phi_2$ such that the total variation of $\nu_1-\nu_2$ is bounded by $\epsilon$, that is,
$$
|\nu_1-\nu_2|(\R^d)<\epsilon.
$$

\begin{lemma}\label{lem_perturbation}
Given two reals $\epsilon_0>\epsilon>0$. Let $\Phi_1,\Phi_2 \subset \mcal{P}(\R^d)$ satisfy $\Phi_1\subseteq_\epsilon\Phi_2$. Suppose that $\Phi_2$ is an $\epsilon_0$-equi-positive family. Then $\Phi_1$ is a $(\epsilon_0-\epsilon)$-equi-positive family.
\end{lemma}

\begin{proof}
Since $\Phi_2$ is an $\epsilon_0$-equi-positive family, there exists $\delta_0>0$ such that for all $x\in[0,1)^d $ and $\mu \in \Phi_2$, there exists $k_{x,\mu }\in \Z^d$ such that
\[
| \hat{\mu}(x+y+k_{x,\mu })| \ge \epsilon_0,
\]
for all $|y|< \delta_0$, where $k_{x,\mu }=0$ for $x=0$. Since $\Phi_1\subseteq_\epsilon\Phi_2$,  for each $\nu_1\in \Phi_1$, there exists $\nu_2\in\Phi_2$ such that
$$
|\nu_1-\nu_2|(\R^d)<\epsilon.
$$
Thus, it follows that  for all $\xi\in\R^d$,
$$
|\hat{\nu}_1(\xi)-\hat{\nu}_2(\xi)| \le\big|\int_{\R^d} e^{-2\pi i \xi \cdot x} \D (\nu_1-\nu_2)(x)\big|  <\epsilon.
$$

Therefore, for all $x\in[0,1)^d $ and $\nu_1 \in \Phi_1$, setting
\begin{equation}\label{eq_kv=kvp}
k_{x,\nu_1}=k_{x,\nu_2}\in \Z^d,
\end{equation}
we have that for all $|y|< \delta_0$,
\[
| \hat{\nu}_1(x+y+k_{x,\nu_1 })| \ge | \hat{\nu}_2(x+y+k_{x,\nu_2 })|-| \hat{\nu}_1(x+y+k_{x,\nu_1 })-\hat{\nu}_2(x+y+k_{x,\nu_2 })|\ge\epsilon_0-\epsilon,
\]
 where $k_{x,\nu_1 }=k_{x,\nu_2}=0$ if $x=0$, which implies that $\Phi_1$ is a $(\epsilon_0-\epsilon)$-equi-positive family.
\end{proof}

\section{properties of equivalence sequences}\label{sec_equ}

Let $X$ be a random vector on a probability space $(\Omega,\mscr{F},P)$. The image measure of $P$ under $X$, denoted by $\mu$, is called the \emph{distribution} of $X$, i.e., $\mu(E) = P(X^{-1}E)$ for all Borel subsets $E \sse \R^d$. See \cite{Kallenberg} for the background reading.

First, we define the equivalence for two sequences of random vectors.
\begin{definition}
We call two sequences of random vectors $\{X_k\}_{k=1}^\f$ and $\{Y_k\}_{k=1}^\f$ on a probability space $(\Omega,\mscr{F},P)$ are {\it equivalent} if $\sum_{k=1}^{\f}P(X_k\ne Y_k)<\f$.
\end{definition}

Let $\{\eta_k\}_{k=1}^\infty$ be a sequence in $\mcal{P}(\R^d)$. By the existence theorem of product measures, there exists a probability space $(\Omega, \mathscr{F},P)$ and a sequence of independent random vectors $\{ X_k \}_{k=1}^\f$ such that the distribution of $X_k$ is $\eta_k$  for each $k \ge 1$.
It is clear that  $\eta_1 * \eta_2 * \cdots *\eta_n$ is the distribution of $X_1 + X_2 + \cdots + X_n$.
Thus, the existence of   $\eta_1 * \eta_2* \cdots $ is equivalent to the convergence of the series $\sum_{k=1}^\infty X_k$ in distribution. Since the convergence in distribution and the almost sure convergence of the sum of independent random vectors are equivalent,  it is sufficient to study the  almost sure convergence of the series $\sum_{k=1}^\infty X_k$. See \cite{Jessen-Wintner-1935} for details.  We  apply this to prove Theorem \ref{thm_equivalent} by using Borel-Cantelli Lemma.

Given two sequences of finite sets $\{A_k\}_{k=1}^\infty$ and $\{A'_k\}_{k=1}^\infty$ in $\R^d$, correspondingly, there exists two sequence of probability measures $\{\delta_{A_k}\}_{k=1}^\infty$ and $\{\delta_{A'_k}\}_{k=1}^\infty$. By the existence theorem of product measures, there exists a probability space $(\Omega, \mathscr{F},P)$ and two sequences of independent random vectors $\{ X_k \}_{k=1}^\f$ and $\{Y_k\}_{k=1}^\f$ such that for each $k \ge 1$, the distributions of $X_k$ and $Y_k$ are $\delta_{A_k}$ and $\delta_{A'_k}$ respectively.
Then we have
$$
P(X_k\ne Y_k)\ge\max\bigg\{\frac{\#(A'_k\setminus A_k)}{\# A'_k},\frac{\#(A_k\setminus A'_k)}{\# A_k}\bigg\},
$$
and it implies that $\{A_k\}_{k=1}^\f$ and $\{A'_k\}_{k=1}^\f$ are equivalent  if $\{X_k\}_{k=1}^\f$ and $\{Y_k\}_{k=1}^\f$are equivalent. In fact, we may construct a special probability space and two special sequences of random vectors such that
$$
P(X_k\ne Y_k)=\max\bigg\{\frac{\#(A'_k\setminus A_k)}{\# A'_k},\frac{\#(A_k\setminus A'_k)}{\# A_k}\bigg\},
$$
in the proof of Theorem \ref{thm_equivalent}.

\begin{proof}[Proof of Theorem \ref{thm_equivalent}]
First, we construct two sequences $\{ X_k \}_{k=1}^\f$ and $\{Y_k\}_{k=1}^\f$ of random vectors on the probability space $([0,1],\mathscr{B},\mbb{L}_{[0,1]})$ such that their distributions are $\{\delta_{A_k}\}_{k=1}^\infty$ and $\{\delta_{A'_k}\}_{k=1}^\infty$ respectively.

Given an integer $k>0$. Without loss of generality, we assume that
$$
A_k=\{a_1,a_2,\dots,a_m\}, \quad A'_k=\{a'_1,a'_2,\dots, a'_n\},
$$
where $m\leq n$ and $\ a_i=a'_i$ for all $1\le i\le \# (A_k\cap A'_k)$.
We write
\begin{equation} \label{def_RVX}
X_k(x)=
\begin{cases}
	a_i & \text{if} \ x\in[\frac{i-1}{m},\frac{i}{m});\\
	a_m   & \text{if} \ x=1.
\end{cases}
\end{equation}
Let $Y_k$ be a measurable fuction defined on $\{[0,1],\mathscr{B},\mbb{L}_{[0,1]}\}$ satisfying that $\mbb{L}_{[0,1]}(Y_k=a'_i)=\frac{1}{n}$ for  $1\le i \le n$ such that
\begin{equation} \label{def_RVY}
Y_k(x)=a_i', \ \text{if} \ x\in\bigg[\frac{i-1}{m},\frac{i-1}{m}+\frac{1}{n}\bigg), \quad i =1,\ldots, m.
\end{equation}

It is clear that  the distribution of $X_k$ and $Y_k$ are $\delta_{A_k}$ and $\delta_{A'_k}$ respectively, and
$$
\mbb{L}_{[0,1]}(X_k\ne Y_k)=\max\bigg\{\frac{\#(A'_k\setminus A_k)}{\# A'_k},\frac{\#(A_k\setminus A'_k)}{\# A_k}\bigg\}.
$$
For each $k>0$,  since $\{A_k\}_{k=1}^\f$ and $\{A'_k\}_{k=1}^\f$ are equivalent,  we have
$$
\sum_{k=1}^{\f}\mbb{L}_{[0,1]}(X_k\ne Y_k)=\sum_{k=1}^{\f}\max\bigg\{\frac{\#(A'_k\setminus A_k)}{\# A'_k},\frac{\#(A_k\setminus A'_k)}{\# A_k}\bigg\}<\f.
$$
By Borel-Cantelli Lemma\cite{Kallenberg}, it follows  that $\mbb{L}_{[0,1]}\{\sum_{k=1}^{\f}X_k \ \text{converges}\}=1$ if and only if $\mbb{L}_{[0,1]}\{\sum_{k=1}^{\f}Y_k \ \text{converges}\}=1$.  This implies  that the infinite convolution $\nu$ of $\{A_k\}_{k=1}^\infty $ exists if and only if the infinite convolution $\nu'$ of $\{A'_k\}_{k=1}^\infty $ exists.
\end{proof}

\begin{proof}[Proof of Corollary \ref{cor_equivalent}]
Since $\{B_k\}_{k=1}^\f$ and $\{B'_k\}_{k=1}^\f$ are equivalent, it is clear that the two sequences $\{(R_k\cdots R_1)^{-1}B_k\}_{k=1}^\f$ and $\{(R_k\cdots R_1)^{-1}B'_k\}_{k=1}^\f$ are also equivalent.  By Theorem \ref{thm_equivalent},  the infinite convolution $\mu$ of $\{(R_k,B_k)\}_{k=1}^\infty $ exists if and only if the infinite convolution $\mu'$ of $\{(R_k,B'_k)\}_{k=1}^\infty $ exists.

It remains to show the equi-positivity. Similar to \eqref{def_RVX} and \eqref{def_RVY}, we  construct two sequences $\{X_k\}_{k=1}^\f$ and $\{Y_k\}_{k=1}^\f$ of random vectors  on $\{[0,1],\mathscr{B},\mbb{L}_{[0,1]}\}$ such that  their distributions are $\{\delta_{(R_k\cdots R_1)^{-1}B_k}\}_{k=1}^\infty$ and $\{\delta_{(R_k\cdots R_1)^{-1}B'_k}\}_{k=1}^\infty$ respectively, and
$$
\sum_{k=1}^{\f}\mbb{L}_{[0,1]}(X_k\ne Y_k)=\sum_{k=1}^{\f}\max\bigg\{\frac{\#(B'_k\setminus B_k)}{\# B'_k},\frac{\#(B_k\setminus B'_k)}{\# B_k}\bigg\}<\f.
$$

Assume that $\mu$ exists. The infinite convolutions $\nu_{>k}$ and $\nu'_{>k}$ given by \eqref{def_nu_n} exist for all integers $k>0$.
Given an arbitrary real $\epsilon>0$. Since
$$
\sum_{k=1}^{\f}\max\bigg\{\frac{\#(B'_k\setminus B_k)}{\# B'_k},\frac{\#(B_k\setminus B'_k)}{\# B_k}\bigg\}<\f,
$$
there exists an integer $K>0$ such that
$$
\sum_{k=K}^{\f}\max\bigg\{\frac{\#(B'_k\setminus B_k)}{\# B'_k},\frac{\#(B_k\setminus B'_k)}{\# B_k}\bigg\}<\frac{\epsilon}{2}.
$$
It follows that for all $k_0\ge K$,
$$
|\nu_{>k_0}-\nu'_{>k_0}|(\R^d)\le\sum_{k=k_0}^{\f}2\cdot\mbb{L}_{[0,1]}(X_k\ne Y_k)\le 2\sum_{k=K}^{\f}\max\bigg\{\frac{\#(B'_k\setminus B_k)}{\# B'_k},\frac{\#(B_k\setminus B'_k)}{\# B_k}\bigg\}<\epsilon,
$$
and it implies that
$\{\nu_{>k}\}_{k=K}^\infty\subseteq_\epsilon\{\nu'_{>k}\}_{k=K}^\infty$ and
$\{\nu'_{>k}\}_{k=K}^\infty\subseteq_\epsilon\{\nu_{>k}\}_{k=K}^\infty$.

By Lemma \ref{lem_perturbation} and the arbitrariness of $\epsilon$, $\mu$ is an equi-positive measure if and only if $\mu'$ is an equi-positive measure.
\end{proof}

\begin{proof}[Proof of Theorem \ref{thm_common spectrum}]
(1) It is a direction consequence of    Corollary \ref{cor_equivalent} that   the infinite convolution $\mu'$ of $\{(R_k,B_k')\}_{k=1}^\infty $ exists and is an equi-positive measure.

(2) Since $B'_k\equiv B_k \pmod{R_k\Z^{d}}$ for all $k>0$, by Lemma \ref{lem_hadamard}, there exists a sequence $\{L_k\}_{k=1}^\f$ such that both
$(R_k,B_k,L_k)$ and $(R_k,B'_k,L_k)$ are Hadamard triples for all $k>0$.

Since  the infinite convolution $\mu$ of $\{(R_k,B_k)\}_{k=1}^\infty $ is an equi-positive measure, there exists a  sequence $\{n_j\}$ such that $\{\nu_{>n_j}\}_{j=1}^\infty$ is an equi-positive family. By   Corollary \ref{cor_equivalent} and \eqref{eq_kv=kvp} in the proof of Lemma \ref{lem_perturbation}, there exists $J>0$ such that $\{\nu'_{>n_j}\}_{j=J}^\infty$ is also an equi-positive family, and for all $j\ge J$ and all $x\in[0,1)^d$,
$$
k_{x,\nu_{>n_j}}=k_{x,\nu'_{>n_j}},
$$

Note that the spectrum $\Lambda$ of $\mu$  depends entirely on the equi-positive family; see the remark after Theorem \ref{thm_equi_spectral}.
By the same argument in  Theorem \ref{thm_equi_spectral}, the set $\Lambda=\bigcup_{j=1}^\f\Lambda_j$ is a spectrum of $\mu$ where
$$
\Lambda_j=\Lambda_{j-1}+\mbf R_{0,m_{j-1}}\{\lambda+\mbf R_{m_{j-1},m_j}^Tk_{\lambda,j}:\lambda\in\mbf L_{m_{j-1},m_j}\}
$$
is constructed with respect to the equi-positive family $\{\nu_{>n_j}\}_{j=J}^\infty$.  Since $(R_k,B_k',L_k)$ forms a Hadamard triple for all $k>0$ and  $\{\nu_{>n_j}'\}_{j=J}^\infty$ is also an equi-positive family. Thus, the set $\Lambda$ is also a spectrum of $\mu'$. Therefore  both $(\mu,\Lambda)$ and $(\mu',\Lambda)$ are spectral pairs.
\end{proof}

\section{Existence of infinite convolutions}\label{sec_existence}
In this section, we study the existence of infinite convolutions $\mu$ of $\{( R_k,B_k)\}_{k=1}^\infty $ by using Kolmogorov’s three series theorem, which is inspired by  Jessen and Wintner's work in~\cite{Jessen-Wintner-1935}, where they provided some sufficient and necessary conditions for the convergence of infinite convolutions in a general setting. In  \cite{Miao-2022},  Li, Miao and Wang also applied this idea and obtained a necessary and sufficient condition for the existence of infinite convolutions defined as \eqref{def_ica}.

For $\eta \in \mathcal{P}(\R^d)$, we define $$ E(\eta) = \int_{\R^d} x \D \eta(x), \qquad V(\eta) = \int_{\R^d} |x - c(\eta)|^2 \D \eta(x). $$
It is easy to check that $ V(\eta) = \int_{\R^d} |x|^2 \D \eta(x) - |E(\eta)|^2. $  Given $r>0$, we define a new Borel probability measure $\eta_r$ by
\begin{equation}\label{equ_eta_r}
	\eta_r(E) = \eta\big( E \cap D(r) \big) + \eta\big( \R^d \setminus D(r) \big)\delta_0(E),
\end{equation}
for every Borel subset $E \sse \R^d$,  where $D(r)$ is the closed ball with centre at $0$ and radius $r$.
The following theorem is from\cite[Theorem 34]{Jessen-Wintner-1935}, and it may be regarded as a distribution  version of Kolmogorov’s three series theorem.
\begin{theorem}\label{thm_three-series-theorem}
 Let $\{\eta_k\}_{k=1}^\infty$ be a sequence in $\mcal{P}(\R^d)$. Given $r>0$, let $\eta_{k,r}$ be defined by \eqref{equ_eta_r} for each $k>0$.
Then the infinite convolution $\eta_1 *\eta_2 *\cdots $ exists if and only if the following three series converge:
$$ (i) \sum_{k=1}^{\f} \eta_k\big( \R^d \setminus D(r)\big),\quad (ii) \sum_{k=1}^{\f} E(\eta_{k,r}), \quad (iii) \sum_{k=1}^{\f} V(\eta_{k,r}). $$
\end{theorem}

Next we  prove Theorem \ref{thm_RBC_exists} by applying Theorem \ref{thm_three-series-theorem}.
\begin{proof}[Proof of Theorem \ref{thm_RBC_exists}]
Since $\{(R_k,B_k)\}_{k=1}^\infty $ satisfies RBC,  we have
$$
\sum_{k=1}^{\f}\max\bigg\{\frac{\#(B_{k,1}\setminus B_k)}{\# B_{k,1}},\frac{\#(B_k\setminus B_{k,1})}{\# B_k}\bigg\}=\sum_{k=1}^{\infty} \frac{\# B_{k,2}}{\# B_k} <\f,
$$
where $B_{k,1}=B_k \cap R_k[-\frac{1}{2},\frac{1}{2})^d$ and $B_{k,2}=B_k \backslash B_{k,1}$, and this implies that $\{B_k\}_{k=1}^\infty$ and $\{B_{k,1}\}_{k=1}^\infty$ are equivalent.

 Fix $r=1$. For each $k>0$, we write $\eta_k=\delta_{(R_kR_{k-1}\cdots R_1)^{-1}B_{k,1}}$, and let $\eta_{k,r}$ be defined by \eqref{equ_eta_r}.
Since $\{(R_k,B_k)\}_{k=1}^\infty $ satisfies uniform contractive condition, there exists $0<\epsilon<1$ such that
$\Vert R^{-1}_k\Vert<\epsilon$ for all $k>0$. Hence  there exists $k_0>0$ such that  for all $k\ge k_0$,
$$
\max\{|(R_kR_{k-1}\cdots R_1)^{-1}b|:b\in B_{k,1}\}<1.
$$
and it implies that $\eta_{k,r}=\eta_k$ for all $k\ge k_0$. Hence, we have $\sum_{k=k_0}^{\f} \eta_k\big( \R^d \setminus D(r)\big)=0 $
and
\begin{align*}
\sum_{k=k_0}^{\f} V(\eta_{k,r})
&=\sum_{k=k_0}^{\f} V(\eta_{k})\le\sum_{k=k_0}^{\f}\int_{\R^d} |x|^2 \D \eta_k(x)\\
&\le\sum_{k=k_0}^{\f}\max\{|(R_kR_{k-1}\cdots R_1)^{-1}b|^2:b\in B_{k,1}\}\\
&\le\sum_{k=k_0}^{\f}\Vert(R_{k-1}\cdots R_1)^{-1}\Vert^2\cdot\frac{d}{4}\\
&=\frac{d}{4}\cdot\sum_{k=k_0}^{\f}\epsilon^{2(k-1)}<\f.
\end{align*}

Since $B_{k,1}\subseteq R_k[-\frac{1}{2},\frac{1}{2})^d$,  there exits $c_k\in[-\frac{1}{2},\frac{1}{2})^d$ such that
$$
\frac{1}{\#B_{k,1}}\sum_{b\in B_{k,1}}(R_kR_{k-1}\cdots R_1)^{-1}b=(R_{k-1}\cdots R_1)^{-1}c_k,
$$
for all $k>0$.  Since $\Vert R^{-1}_k\Vert<\epsilon$ for all $k>0$, it is clear that $\sum_{k=1}^{\f}(R_{k-1}\cdots R_1)^{-1}c_k$ converges, and we obtain that
$$
\sum_{k=k_0}^{\f} E(\eta_{k,r})=\sum_{k=k_0}^{\f}\frac{1}{\#B_{k,1}}\sum_{b\in B_{k,1}}(R_kR_{k-1}\cdots R_1)^{-1}b=\sum_{k=k_0}^{\f}(R_{k-1}\cdots R_1)^{-1}c_k
$$
converges.

Therefore, by Theorem \ref{thm_three-series-theorem}, the infinite convolution $\mu'$ of $( R_k,B_{k,1})_{k=1}^\infty$ exists. Since $\{B_k\}_{k=1}^\infty$ and $\{B_{k,1}\}_{k=1}^\infty$ are equivalent, by Theorem \ref{thm_equivalent},  the infinite convolution $\mu$ of $(R_k,B_k)_{k=1}^\infty$ exists.
\end{proof}

\section{Spectrality of infinite convolutions}\label{sec_spectral}

The following conclusion provides a useful estimate to verify the  equi-positivity of infinite convolutions.

\begin{lemma}\label{lem_theta}
Given $\theta\in[0,\pi)$. There exists a constant $r(\theta)\in (0,1)$ such that
\[
\Big|\frac{1}{m}\sum_{j=1}^{m}e^{-ix_j}\Big|\ge r(\theta),
\]
for all integers $m>0$ and all $x_j\in[0,\theta]$, $j=1,\dots,m$.
\end{lemma}
\begin{proof}
	Given $\theta\in(0,\pi)$, for all integers $m>0$ and all $x_j\in[0,\theta]$, $j=1,\dots,m$, it is straightforward that
$$
\Big|\frac{1}{m}\sum_{j=1}^{m}e^{-ix_j}\Big|  =  \Big|\frac{1}{m}\sum_{j=1}^{m}e^{-i(x_j-\frac{\theta}{2})}\Big|  \ge \frac{1}{m}\sum_{j=1}^{m}\cos\bigg(x_j-\frac{\theta}{2}\bigg)  \ge \cos\frac{\theta}{2}.
$$
	Setting $r(\theta)=\cos\frac{\theta}{2}$, the concluions holds.
\end{proof}

\iffalse
\begin{lemma}\label{lem_theta'}
Given $\theta\in[0,\pi)$ and $c\in[ 0,1) $.  There exists a small $\Delta>0$ and  $\theta^{'}\in(\theta,\pi)$ such that for all $0<\delta <\Delta$, we have that 	
$$
(1+\delta)\theta \leq \theta^{'} \quad \textit{and } \quad  r(\theta^{'})\ge c\cdot r(\theta).
$$
\end{lemma}
\fi

\begin{proof}[Proof of Theorem \ref{thm_PCC_spectral}]
Since the  sequence $\{( R_k,B_k)\}_{k=1}^\infty $  of admissible pairs satisfies {\it RBC} and uniform contractive condition, by Theorem \ref{thm_RBC_exists},  the infinite convolution $\mu$ exists.

For each integer $k>0$,  we choose $B'_k\subseteq R_k[-\frac{1}{2},\frac{1}{2})^d$ such that $B'_k\equiv B_k \pmod{R_k\Z^d}$, and  by (ii) in Lemma \ref{lem_hadamard},  $(R_k,B'_k)$ is an admissible pair.
Since $\{(R_k,B_k)\}_{k=1}^\infty $ satisfies RBC,  it implies that
$$
\sum_{k=1}^{\f}\max\bigg\{\frac{\#(B'_k\setminus B_k)}{\# B'_k},\frac{\#(B_k\setminus B'_k)}{\# B_k}\bigg\}=\sum_{k=1}^{\infty} \frac{\# B_{k,2}}{\# B_k} <\f,
$$
i.e.,  $\{B_k\}_{k=1}^\infty$ and $\{B'_k\}_{k=1}^\infty$ are equivalent. Note that $B_{k,1}\subseteq B'_{k,1}$ for all $k>0$.

Since the subsequence $\{( R_{n_k},B_{n_k})\}_{k=1}^\infty  $ satisfies  PCC with some $l\in(0,1)$, i.e.,
\begin{equation}\label{cdnpcc1}
\sup\big\{\big(R_{n_k}^{-1}(x_1-x_2)\big)\cdot\xi:x_1,x_2\in\bar{B}\big(0,\frac{\sqrt{d}}{2}\big)\big\}<1-l,
\end{equation}
for all $\xi\in [-1,1]^d$ and
\begin{equation}\label{cdnpcc2}
\sum_{k=1}^{\f}\frac{\# B_{{n_k},2}^{l}}{\# B_{n_k}}<\f,
\end{equation}
where $B_{{n_k},1}^{l}=B_k\cap \{b:(R_{n_k}^{-1}b)\cdot \xi\in(-\frac{1-l}{2},\frac{1-l}{2}) \ \text{for all}\ \xi\in[-1,1]^d\}$ and $B_{{n_k},2}^{l}=B_{n_k}\backslash B_{{n_k},1}^{l},$ we have that $B_{{n_k},1}^{l}\subseteq B_{{n_k},1}$,  $B_{{n_k},1}^{l}\subseteq B_{{n_k},1}^{'l}$ and
\[
\sum_{k=1}^{\f}\frac{\# B_{{n_k},2}^{'l}}{\# B_{n_k}}\le \sum_{k=1}^{\f}\frac{\# B_{{n_k},2}^{l}}{\# B_{n_k}}<\f,
\]
and it implies that $\{( R_{n_k},B'_{n_k})\}_{k=1}^\infty $ satisfies  PCC with $l$, where  $B_{{n_k},2}^{l}=B_{n_k}\backslash B_{{n_k},1}^{l}.$

Therefore, by Theorem \ref{thm_equivalent} and Theorem \ref{thm_equi_spectral}, it is sufficient to show the infinite convolution $\mu'$ of  $\{(R_k,B'_k)\}_{k=1}^\infty $ is an equi-positive measure. Recall that for each $k>1$,
\[
\nu'_{>n_k-1}=\delta_{R_{n_k}^{-1}B'_{n_k}}\ast\delta_{(R_{n_k+1}R_{n_k})^{-1}B'_{n_k+1}}\ast\dots\ast\delta_{(R_{n_k+j}R_{n_k+j-1}\dots R_{n_k})^{-1}B'_{n_k+j}}\ast\cdots,
\]
and
\[
{\widehat{\nu'}}_{>n_k-1}(\xi)=\prod_{j=0}^{\infty}M_{B'_{n_k+j}}((R_{n_k+j}R_{n_k+j-1}\dots R_{n_k})^{-T}\xi)=\prod_{j=0}^{\infty}M_{B'_{n_k+j}}(\xi_j),
\]	
where $\xi_j=(R_{n_k+j}R_{n_k+j-1}\dots R_{n_k})^{-T}\xi$ for each $j\geq 0$.
The key is to give appropriate estimate for the lower bounds of $\prod_{j=0}^{\infty}|M_{B'_{n_k+j}}(\xi_j)|$.	

Since $\{(R_k,B_k)\}_{k=1}^\infty $ satisfies uniform contractive condition, there exists $0<\epsilon<1$ such that
$\Vert R^{-1}_k\Vert<\epsilon$ for all $k>0$.
Arbitrarily choosing $\delta\in(0,\frac{1}{4})$,  for all $\xi\in[-\frac{1}{2}-\delta,\frac{1}{2}+\delta]^d$,  we have
\[
|2\pi b\cdot\xi_j|\le\Vert (R_{n_k+j-1}\dots R_{n_k})^{-1}\Vert|2\pi (R_{n_k+j}^{-1}b)\cdot\xi|<\epsilon^{j-1}\cdot2\pi\cdot\frac{\sqrt{d}}{2}\cdot\frac{3\sqrt{d}}{4}=\frac{3d\pi}{4}\cdot\epsilon^{j-1},
\]
for all $b\in B'_{n_k+j} \subseteq R_k[-\frac{1}{2},\frac{1}{2})^d$ and all integers $j>0$.	
Therefore, there exists $J>0$ independent of $\delta$ and $\xi$ such that
$$
|2\pi b\cdot\xi_j|\le\epsilon^{j-J},
$$
for all $b\in B'_{n_k+j}$, $\xi\in[-\frac{1}{2}-\delta,\frac{1}{2}+\delta]^d$ and all $j\ge J$.

First, we estimate $\prod_{j=J}^{\infty}|M_{B'_{n_k+j}}(\xi_j)|$.  For each $j\ge J$, we have
$$
\left|M_{B'_{n_k+j}}(\xi_j)\right|
=\Big|\frac{1}{\#B'_{n_k+j}}\sum_{b\in B'_{n_k+j}}e^{-2\pi i b\cdot\xi_j}\Big|
\ge \Big|\frac{1}{\#B'_{n_k+j}}\sum_{b\in B_{n_k+j}^{'}}\cos2\pi b\cdot \xi_j\Big|
\ge \cos\epsilon^{j-J}>0.
$$
Since $\lim_{x\to 0}\frac{-\ln\cos x}{x^2}=\frac{1}{2}$ and $\sum_{j=0}^{\f}\epsilon^{2j}$ is convergent, the series  $\sum_{j=0}^{\f}\ln\cos\epsilon^j$ converges. Let $\sum_{j=0}^{\f}\ln\cos\epsilon^j=C$  for some $C\in\R$, and  we obtain the lower bound
\begin{equation}\label{eq_j>J}
\prod_{j=J}^{\infty}\left|M_{B'_{n_k+j}}(\xi_j)\right|\ge\prod_{j=J}^{\f}( \cos\epsilon^{j-J})\geq e^{C}.	
\end{equation}

Next, we estimate $\prod_{j=1}^{J-1}|M_{B'_{n_k+j}}(\xi_j)|$.
Since $B'_k\subseteq R_k[-\frac{1}{2},\frac{1}{2})^d$ and $\Vert R^{-1}_k\Vert<\epsilon$ for all $k>0$, we have $(R_{n_k+j}\dots R_{n_k+1})^{-1}B_{n_k+j}\subseteq\bar{B}\big(0,\frac{\sqrt{d}}{2}\big).$ Arbitrarily choosing  $\delta\in(0,\frac{l}{4(1-l)})$,  for all $\xi\in[-\frac{1}{2}-\delta,\frac{1}{2}+\delta]^d$, by \eqref{cdnpcc1}, we have
\begin{align*}
&\max\{2\pi (R_{n_k+j}\dots R_{n_k+1}R_{n_k})^{-1}(b_1-b_2)\cdot\xi:b\in B'_{n_k+j}\}\\
&\le(1+2\delta)\pi\max\big\{R_{n_k}^{-1}(x_1-x_2)\cdot\frac{2\xi}{1+2\delta}:x_1,x_2\in\bar{B}\big(0,\frac{\sqrt{d}}{2}\big)\big\}\\
&<(1+2\delta)(1-l)\pi\\
&<(1-\frac{l}{2})\pi,
\end{align*}
for all $k,j>0$.
Therefore by Lemma \ref{lem_theta}, we have $|M_{B'_{n_k+j}}(\xi_j)|\ge r((1-\frac{l}{2})\pi)$ for all $k,j>0$, and we obtain that
\begin{equation}\label{eq_0<j<J}
\prod_{j=1}^{J-1}|M_{B'_{n_k+j}}(\xi_j)|\ge r\big((1-\frac{l}{2})\pi\big)^{J-1},
\end{equation}
for all $k>0$.

It remains to estimate   $M_{B_{n_k}}(\xi_0)$. Given $\delta\in(0,\frac{l}{4(1-l)})$, for $\xi\in[-\frac{1}{2}-\delta,\frac{1}{2}+\delta]^d$,  by \eqref{cdnpcc1}, we have
$$
|2\pi R_{n_k}^{-1}b\cdot \xi|\le(1+2\delta)\pi|R_{n_k}^{-1}b\cdot \frac{2\xi}{1+2\delta}|<(1+2\delta)\pi\cdot\frac{1-l}{2}<\frac{(1-\frac{l}{2})\pi}{2},
$$
for all $b\in B_{n_k,1}^{'l}$. Then, by Lemma \ref{lem_theta}, we have
\[
\Big|\frac{1}{\# B_{n_k,1}^{'l}}\sum_{b\in B_{n_k,1}^{'l}}e^{-2\pi i R_{n_k}^{-1}b\cdot\xi}\Big|\ge r\big((1-\frac{l}{2})\pi\big).
\]

For all sufficiently small $a\in(0,1)$ satisfying $(1-a)r((1-\frac{l}{2})\pi)-a>a$, by \eqref{cdnpcc2}, there exists $K>0$ such that for all $k>K$,
\[
\frac{\# B_{n_k,2}^{'l}}{\# B'_{n_k}}<a \quad \text{and} \quad \frac{\# B_{n_k,1}^{'l}}{\# B'_{n_k}}>1-a.
\]
Thus,  for all $k>K$, we obtain that
\begin{eqnarray}\label{eq_j=0}
\big|M_{B_{n_k}}(\xi_0)\big|
&\ge&\Big|\frac{1}{\#B'_{n_k}}\sum_{b\in B_{n_k,1}^{'l}}e^{-2\pi i R_{n_k}^{-1}b\cdot\xi}\Big|-\Big|\frac{1}{\#B'_{n_k}}\sum_{b\in B_{n_k,1}^{'l}}e^{-2\pi i b\cdot\xi_0}-\frac{1}{\#B'_{n_k}}\sum_{b\in B'_{n_k}}e^{-2\pi i b\cdot\xi_0}\Big|           \nonumber\\
&\ge&\Big(\frac{\# B_{n_k,1}^{'l}}{\# B'_{n_k}}\Big)\cdot r\big((1-\frac{l}{2})\pi\big)-\frac{\# B_{n_k,2}^{'l}}{\# B'_{n_k}}   \nonumber\\
&>& \Big(1-a\Big)\cdot r\big((1-\frac{l}{2})\pi\big)-a   > a.
\end{eqnarray}

Finally, Choose  $\delta_0\in(0,\min\{\frac{1}{4},\frac{l}{4(1-l)}\})$.  Combining \eqref{eq_j>J}, \eqref{eq_0<j<J} and \eqref{eq_j=0} together, let
\[
\epsilon_0=e^C\cdot r\big((1-\frac{l}{2})\pi\big)^{J-1}\cdot a>0,
\]
and for all $\xi\in[-\frac{1}{2}-\delta_0,\frac{1}{2}+\delta_0]^d$ and all $k>K$, we obtain
\[
|\widehat{\nu'}_{>n_k-1}(\xi)|=|M_{B'_{n_k}}(\xi_0)||\prod_{j=1}^{J-1}|M_{B'_{n_k+j}}(\xi_j)|\prod_{j=J}^{\infty}|M_{B'_{n_k+j}}(\xi_j)|\ge\epsilon_0.
\]
Hence, for all $x\in[-\frac{1}{2},\frac{1}{2})^d$ and all $k>K$, it follows that
\[
|\widehat{\nu'}_{>n_k-1}(x+y)|\ge\epsilon_0,
\]
for all $|y|<\delta_0$.  This implies that  $\{\nu'_{>n_k-1}:k>K\}$ is an equi-positive family,  where $k_{x,\nu'_{>n_k-1}}=0$ for all $x\in[-\frac{1}{2},\frac{1}{2})^d$ and all $k>K$. Thus, by the proof of Corollary \ref{cor_equivalent}, $\{\nu_{>n_k-1}:k>K\}$ is also an equi-positive family, where $k_{x,\nu_{>n_k-1}}=0$ for all $x\in[-\frac{1}{2},\frac{1}{2})^d$ and all $k>K$.  By Theorem \ref{thm_equi_spectral}, $\mu$ is a spectral measure.

Moreover, suppose that $(R_k,B_k,L_k)$ is a Hadamard triple where  $0\in L_k\subseteq R_k^{T}[-\frac{1}{2},\frac{1}{2}]^d$  for all $k>0$.
Recall that the spectrum $\Lambda$ of $\mu$  depends entirely on the equi-positive family, and by the same argument in  Theorem \ref{thm_equi_spectral}, the set $\Lambda=\bigcup_{j=1}^\f\Lambda_j$ is a spectrum of $\mu$ where
$$
\Lambda_j=\Lambda_{j-1}+\mbf R_{0,m_{j-1}}^T\{\lambda+\mbf R_{m_{j-1},m_j}^Tk_{\lambda,j}:\lambda\in\mbf L_{m_{j-1},m_j}\}
$$
is constructed with respect to the equi-positive family  $\{\nu'_{>n_k-1}:k>K\}$. Since $L_{m_{j-1},m_j}\subseteq R_{m_{j-1},m_j}^T[-\frac{1}{2},\frac{1}{2})^d$ for all $j>0$, and $k_{\lambda,j}=0$ for all $\lambda\in\mbf L_{m_{j-1},m_j}$ and $j>0$. Then we have that
$$
\Lambda = \bigcup_{k=1}^\f \{L_1+R_1^TL_2+\cdots+(R_{k-1}\cdots R_1)^{T}L_k\}.
$$
\end{proof}

\section*{Acknowledgements}
The authors  wish to thank Prof. Lixiang An and Prof. Xinggang He for their helpful comments.

\end{document}